\documentclass[11pt]{amsart}

\usepackage{amsmath,amsthm,dsfont}
\usepackage{txfonts}
\usepackage{amssymb}
\usepackage{graphicx}
\usepackage{cite}
\usepackage{pstricks,pst-all}
\usepackage[utf8]{inputenc}
\usepackage{subfigure}
\usepackage[T1]{fontenc}
\usepackage{multirow}
\usepackage{listings}
\usepackage{lmodern}
\usepackage{float}
\usepackage{color}
\usepackage{mathtools}
\usepackage{graphics}
\usepackage[english]{babel}
\usepackage{url}
\usepackage{grffile} 
\newtheorem{theorem}{Theorem}[section]
\newtheorem{corollary}[theorem]{Corollary}
\newtheorem{lemma}[theorem]{Lemma}
\newtheorem{proposition}[theorem]{Proposition}
\theoremstyle{definition}
\newtheorem{definition}[theorem]{Definition}

\newtheorem{remark}[theorem]{Remark}
\newcommand{\dd}{\mathrm{d}}
\newcommand{\e}{\mathrm{e}}

\newcommand{\PP}[2]{{\mathbb{P}}}

\renewcommand{\P}{\mathbb{P}}
\newcommand{\N}{\mathbb{N}}
\newcommand{\R}{\mathbb{R}}
\newcommand{\FF}{\mathcal{F}}

\numberwithin{equation}{section}

\title[Noise-induced strong stabilization ]{Noise-induced strong stabilization}

\author[M. Leimbach]{Matti Leimbach}
\address[M. Leimbach]{Technische Universit\"at Berlin, MA 7-5, Str.~des 17.~Juni 136, 10623 Berlin, Germany}
\email{{\tt leimbach.matti@googlemail.com}}

\author[J.C. Mattingly]{Jonathan C. Mattingly}
\address[J.C. Mattingly]{Duke University, Box 90320, Durham NC, USA}
\email{\tt jonm@math.duke.edu}

\author[M. Scheutzow]{Michael Scheutzow}
\address[M. Scheutzow]{Technische Universit\"at Berlin, MA 7-5, Str.~des 17.~Juni 136, 10623 Berlin, Germany}
\email{\tt ms@math.tu-berlin.de}



\begin{document}

\begin{abstract}
  We consider a 2-dimensional stochastic differential equation in polar coordinates depending on several parameters. We show that if these parameters belong to a specific regime then the deterministic
  system explodes in finite time, but the random dynamical system corresponding to the stochastic equation is not only strongly complete but even admits a random attractor. 
\end{abstract}

\maketitle


\section{Introduction}
It is known that the addition of noise can stabilize an explosive ordinary differential equation (ODE) such that it becomes a non-explosive stochastic differential equation (SDE).
For examples, see \cite{Scheutzow93},\cite{Doering12}, \cite{Kolba11}, \cite{Herzog11},
\cite{HerzogMattingly1}, \cite{HerzogMattingly2} and \cite{Kolba17}.
This phenomenon is often called {\em noise-induced stability} or {\em noise-induced stabilization}, if, in addition, the corresponding Markov process admits an invariant probability measure. 
We investigate whether the noise can  induce an even stronger kind of stability, namely the existence of a random attractor. We call such a phenomenon {\em noise-induced strong stabilization}.
The existence of a random attractor implies non-explosion and the existence of an invariant distribution, but not vice versa, see \cite{ochs99}. Note that
\begin{itemize}
    \item [$\bullet$] {\em Stabilization} implies infinite-time existence and statistical recurrence (given by the invariant probability measure) of the solution of individual trajectories. 
    \item [$\bullet$] {\em Strong stabilization} implies infinite-time existence and statistical recurrence (given by the law of the random attractor) of the set-valued solution of arbitrarily large bounded sets of initial conditions.
\end{itemize}

In \cite{LeimbachScheutzow14} the authors show that a certain family of SDEs, which exhibits noise-induced stabilization as shown in \cite{HerzogMattingly1} and \cite{HerzogMattingly2},
is not strongly complete, i.e.~there exist (random) initial conditions for which the 
solutions explode in finite time. In particular, there is no random attractor.
In this paper, we provide a positive answer to the question whether noise-induced strong stabilization is possible. Our result, contained in this note,  seems the first time  noise-induced strong stabilization has been rigorously proven.

We consider the following 2-dimensional SDE in polar coordinates
\begin{equation}
\label{eq:driving}
\begin{split}
\mathrm{d}r_t &=\left(-r_t^w\cos^2(\phi_t)+r^v_t\right)\mathrm{d}t \text{,}\\
\mathrm{d}\phi_t &= -r_t^\gamma\cos^2(\phi_t)\mathrm{d}t +\sigma \mathrm{d}W_t \text{,}
\end{split}
\end{equation}
where $w,v,\gamma>0, \sigma \geq 0$ and $(W_t)_{t\geq 0}$ is a standard one-dimensional Brownian motion. Since the equation is $\pi$-periodic in $\phi$, we work with a periodic boundary condition for the angular component $(\phi_t)_{t\geq 0}$. Hence, the state space reduces to $[0,\infty)\times [0,\pi]$. 

Explosion, existence of an invariant distribution or existence of a random attractor highly depend on the {\em radial process} $(r_t)_{t\geq 0}$. Therefore, it is advantageous to work with polar coordinates providing an equation for the radial component.
Our aim is to find values for $w,v,\gamma>0$, such that in the deterministic case, i.e. $\sigma=0$, the solution explodes in finite time for some or almost all initial conditions, while in the stochastic case, i.e. $\sigma>0$, the random dynamical system corresponding to (\ref{eq:driving}) admits a random attractor.

The usual criteria to prove the existence of a random attractor like monotonicity (see \cite{Scheutzow08}, \cite{CS04}, \cite{FGS17b}) or a drift towards the origin  (see \cite{DimiScheu11}) cannot be applied to our system. Instead we will construct an embedded Semi-Markov process which dominates the radial part and which admits an invariant probability measure and thus guarantees existence of a random attractor.

\section{Overview and heuristics}

We begin by discussing, at an informal level, the deterministic dynamics to better illuminate the source of the instability in the deterministic problem and how noise stabilizes the dynamics. We will also contrast this system with those considered in \cite{Herzog11,Kolba11,HerzogMattingly1, HerzogMattingly2} to better understand why a random attractor exists in this problem but not those examples.

\begin{figure}
    \centering
 \includegraphics[width=0.7\textwidth]{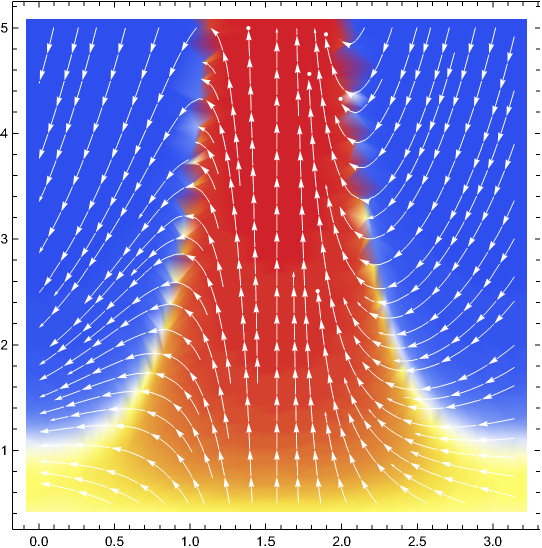}
    \caption{Plot of vector-field generated by the deterministic part of \eqref{eq:driving} with $\gamma=7/4$, $w=3$, and  $v=2$ over $\phi \in [0,\pi]$ (the horizontal axis) and $r \in [1,5]$ (the vertical axis). The red region is where the radial component of the vector field is positive and the blue is where it is negative. The yellow regions are those where the radial component is  close to zero.   }
    \label{fig:det}
\end{figure}

Figure~\ref{fig:det}  gives the vector field described by the deterministic portion of \eqref{eq:driving}. 
We have placed the angular coordinate $\theta$ on the horizontal axis and the radius on the vertical axis. Above all we are interested in vertical paths to infinity as those represent ``blow-up'' paths along which the dynamics can escape to infinity. 
To make the structure clearer, we have colored the regions blue where the vector field points towards smaller $r$ and red if it points to  larger $r$. (Regions where the radial vector field is near zero are colored yellow.) Hence the red channel is the dangerous zone. We will see that trajectories spending too much time in this region will blow up. 
When $\phi=\pi/2$, the deterministic $r$ dynamics reduce to $\dot r= r^v$ and the deterministic $\phi$ dynamics to $\dot \phi=0$. When $v >1$, this dynamics blows up in finite time. 

Just to the right of these $\phi$, the $\phi$ dynamics moves the trajectory towards $\pi/2$ but never crosses the lines $\phi=\pi/2$. In contrast, away from the red region  the trajectories lead towards smaller values of $r$. Since, in this note, we are primarily interested in whether the dynamics escapes to infinite it is useful to consider the structure of the dynamics for $r\gg 1$. Notice that $\dot r$ is negative except when $\cos(\phi)^2 \in[- r^{v - w}, r^{v - w}]$.  We will mainly be interested in the setting where $w>v$ so that the size of this region is shrinking as $r \rightarrow \infty$. Since it is reasonable to approximate $\cos(\phi)$ around $\pi/2$ by $\pi/2-\phi$ and since the deterministic $\phi$ dynamics causes the system to exit the  portion to the left of $\pi/2$, we will be interested in the time spent in the intervals $\pi/2+ [0, r^\frac{v - w}2]$. 

As already noted, the $\phi$ deterministic dynamics will not leave this region. However, when the noise is present it is reasonable to hope that is might leave this critical region fast enough to ensure the system does not blow up.  
Figure~\ref{fig:stochastic} gives numerical evidence supporting this hope. The left panel shows the stochastic dynamics starting from $(\phi,r)$ equal to $(\pi/8,5)$ and $(0,1)$. 
Note, that we have extended the dynamics to the full angular interval $[0,2\pi]$ for visual simplicity.
The right panel shows the corresponding time series of the radius. Both trajectories are using the same noise realization.
\begin{figure}
    \centering
    \includegraphics[width=0.46\textwidth]{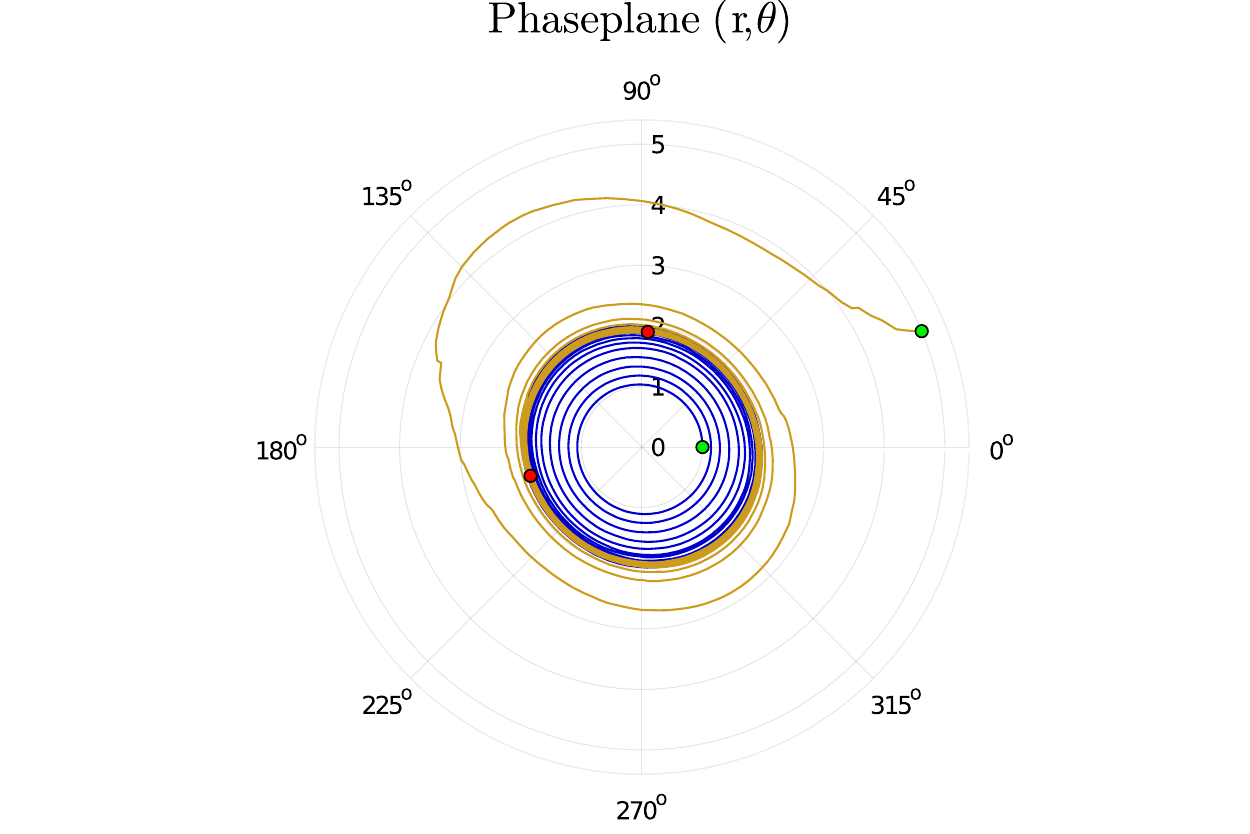}
     \includegraphics[width=0.44\textwidth]{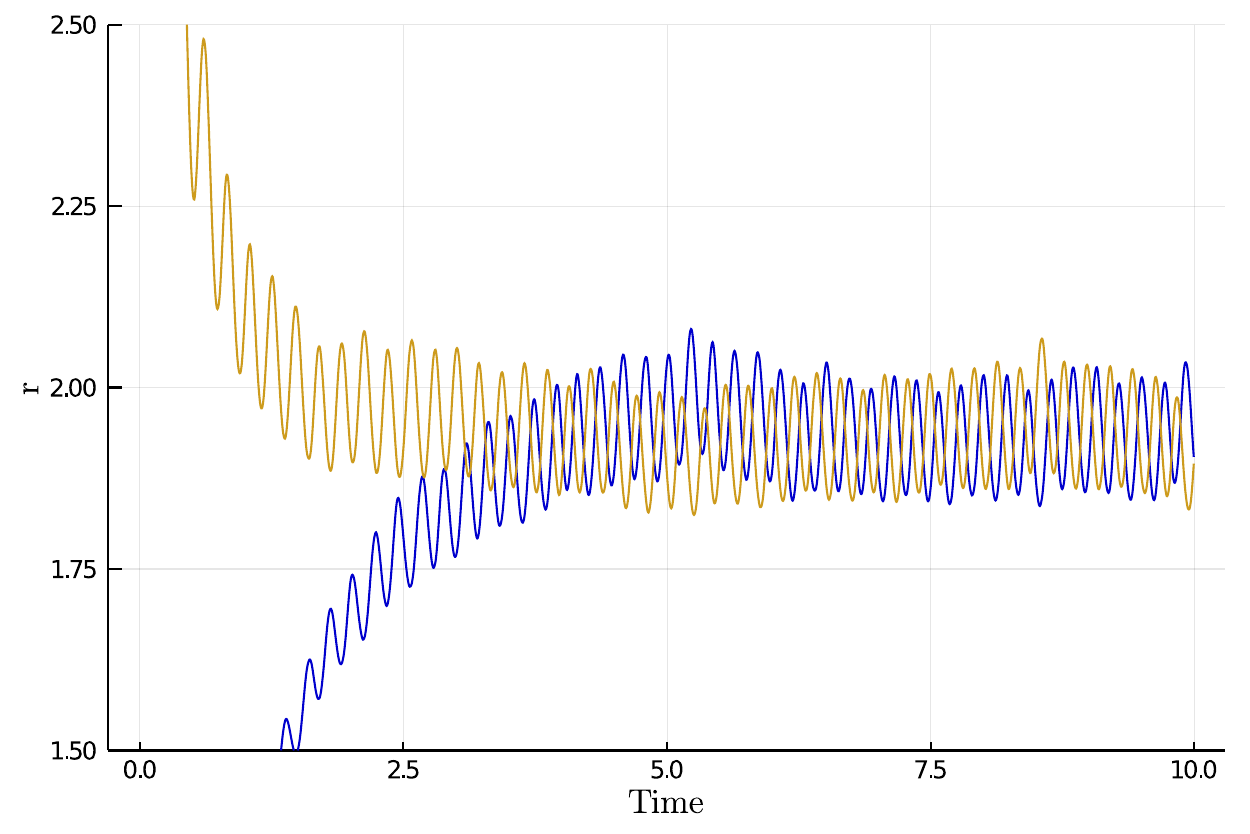}
    \caption{Stochastic dynamics with $\gamma=7/4$, $w=3$, $v=2$ and $\sigma=1$. In the phase plot (left plot), the initial conditions are marked with green dots and the terminal condition with red dots. The right panel plots the radius $r$ versus time for the same two initial conditions.}
    \label{fig:stochastic}
\end{figure}

For large $r$, the noise only becomes relevant when $r^\gamma (\phi-\pi/2)^2$ is order one. 
Depending on the parameters, this region might be completely contained in the interval where the $r$ dynamics are explosive or vice versa. 
In the first case, ($\gamma < \frac{v - w}2)$, we must rely solely on diffusion to cause the dynamics to leave the  critical region where $r$ is exploding. 
When the converse is true, the deterministic $\phi$ dynamics drives the system part of the way into the explosive region until  $r^\gamma (\phi-\pi/2)^2$ is order one and the diffusive dynamics begins to dominate the $\phi$ dynamics. 

Since the $r$ dynamics in the explosive region blows up in finite time, stability of the system turns on whether or not the diffusion (possibly mixed with the deterministic dynamics) can bring $\phi$ through this window before the $r$ dynamics blows up. The tools developed in  \cite{Herzog11,Kolba11,HerzogMattingly1,HerzogMattingly2} give a way to rigorously analyze such a scenario in the context of the one point motion\footnote{The one-point motion being the dynamics of a single trajectory generated by the stochastic flow starting from a single point. This should be contrasted with the two-point motion, namely the  dynamics of two different trajectories subject to the same noise, or the entire flow of stochastic diffeomorphisms, which describe the evolution of the entire phase space at once.}.  Here we are less interested in getting a sharp condition on when the one-point motion is stable but rather prove that for some parameter range  there exists a random attractor. While some of the approximation ideas carry over, we are required to develop estimates which control the trajectories of entire closed sets under the stochastic flow.

As already mentioned, it was shown in \cite{LeimbachScheutzow14} that the planar systems analyzed in  \cite{Herzog11,Kolba11,HerzogMattingly1,HerzogMattingly2} do not possess a random attractor despite having stable one-point motions. The difference, between our current system and those, lies in the properties of the explosive trajectory. In \cite{Herzog11,Kolba11,HerzogMattingly1,HerzogMattingly2} the explosive trajectories are unstable on both sides in the angular variable; hence a small perturbation  of positive or negative $\phi$ would cause the system to quickly leave the unstable trajectory in the angular direction of the perturbation. Thus, points which are perturbed to one side on the unstable trajectory head in a microscopically different direction than those perturbed on to the other side. Since the stochastic flow map is continuous in the initial conditions for short times, there is always a point which rests on the unstable trajectory; and hence, escapes to infinity in finite time. This implies that those SDEs do not generate a flow of stochastic diffeomorphisms which are defined on the entire plane and that some points of the plane are always mapped to infinity in finite time for every realization. 

In contrast, the current system has blow-up trajectories which are attracting from one-side and repelling from the other side (in $\phi$). Hence it is reasonable to expect that there is no one trajectory trapped on top of the exploding trajectory. This note proves that  the stochastic flow of diffeomorphisms does not develop points where the solution is not defined and that all compact sets are attracted to a common random absorbing set, namely a random attractor. In light of this discussion, we expect other systems which have similar stochastic  stabilization mechanisms to also possess random attractors. Examples include the systems studied in \cite{Doering12,Kolba17}.

\section{Blow-up in the deterministic case} 
We now return briefly to the deterministic dynamics to prove the blow-up suggested by Figure~\ref{fig:det}.
Consider the deterministic equation on $[0,\infty)\times [0,\pi]$
\begin{equation}
\label{eq:det_driving}
\begin{split}
\dot{r}_t &=-r_t^w\cos^2(\phi_t)+r^v_t\text{,}\\
\dot{\phi}_t &= -r_t^\gamma\cos^2(\phi_t) \text{,}
\end{split}
\end{equation}
where the equation for $\phi$ should be interpreted modulo $\pi$.
For initial conditions of the form $r_0>0$, $\phi_0=\pi/2$ there is {\em blow-up} or {\em explosion (in finite time)} if and only if $v>1$. It is natural to ask
whether the solution for {\em every} initial condition of the form $r_0>0$, $\phi_0\in [0,\pi]$ blows up. We denote the corresponding blow-up time by
$\mathfrak{e}(r_0,\phi_0) \in (0,\infty]$ (meaning that there is no blow-up in case $\mathfrak{e}(r_0,\phi_0) =\infty$).

The following proposition provides a sufficient condition for explosion in finite time, i.e. $\mathfrak{e}(r_0,\phi_0) < \infty$ for all $r_0>0,\phi_0\in[0,\pi]$.
\begin{proposition}
\label{Proposition:ch3_explosion}
If the parameters satisfy $2\gamma > w-v>0$, $v>1$, then  solutions to equation $(\ref{eq:det_driving})$ blow up for all initial conditions of the form $r_0>0$, $\phi_0\in [0,\pi]$. 
\end{proposition}
\begin{proof}
  Note that \eqref{eq:det_driving} has a unique solution for every initial condition $r_0>0$. 
  For $k \in \mathbb{N}_0$, let $R_k:=2^k$ and let $\Phi_k$ be the unique number in the interval $\big(\frac \pi 2,\pi\big]$ such that
  $R_k^{w-v}\cos^2\big(\Phi_k\big)=1/2$ and
  define $B_k:=\big[R_k,\infty\big)\times \big[\pi/2,\Phi_k\big]$. It is straightforward to check that each set $B_k$ is invariant, i.e.~$(r_T,\phi_T) \in B_k$
  implies $(r_t,\phi_t) \in B_k$ for all $t\geq T$ up to the explosion time $\mathfrak{e}(r_0,\phi_0)$ and that every solution with initial condition $r_0>0$
  hits $B_0$ after a finite time.

  Next, we estimate the time to reach the set $B_{k+1}$ from an arbitrary point in $B_k$ from above by some number $u_k$.
  If $\sum_k u_k <\infty$, then we have blow-up from any starting point with $r_0>0$.
  We can choose $u_k=s_k+t_k$, where $s_k$ is an upper bound for the time it takes until $\phi_t \le \Phi_{k+1}$ when $(r_0,\phi_0) \in B_k$   and 
  $t_k$ is an upper bound for the time it takes until $r_t \ge R_{k+1}$ when $(r_0,\phi_0) \in [R_k,R_{k+1}]\times [0,\Phi_{k+1}]$.

  On $[R_k,\infty)\times [\Phi_{k+1},\Phi_k]$, an upper bound for the derivative of $\phi$ is given by
  $w_k:=-R_k^\gamma\cdot \frac 12 \cdot R_{k+1}^{v-w}=-2^{k(\gamma+v-w)} \cdot 2^{v-w-1}$ whence, as $k \to \infty$,
  $$
     s_k=\frac{\Phi_{k+1}-\Phi_k}{w_k}\le \frac{\Phi_k-\frac \pi 2}{|w_k|}\sim \frac{\sqrt{\cos^2(\Phi_k)}}{|w_k|}= 2^{\frac k2 (-2\gamma -v+w)} \cdot 2^{w-v+\frac 12},
  $$
  which is summable since $2\gamma >w-v$.

  On $\big[R_k,R_{k+1}]\times [0,\Phi_{k+1}]$, a lower bound for the derivative of $r$ is given by $v_k:=R_k^v-R_{k+1}^w\cos^2(\Phi_{k+1})=\frac 12 R_k^v$
  whence
  $$
  t_k=(R_{k+1}-R_k)/v_k\le 2 \cdot 2^{k(1-v)},
  $$
  which is summable since $v>1$, so the proof of the proposition is complete.
\end{proof}

\section{Random dynamical systems and random attractors}
\label{subsec:stoch_flow}

We recall the concepts of a random dynamical system (RDS) and a random attractor. 
We  restrict ourselves to the case in which the state space is $\R^d$ and time is continuous.

\begin{definition}
Let $(\Omega,\FF,\P)$ be a probability space and assume that 
$\vartheta:\R \times \Omega\to \Omega$ is measurable, preserves $\P$ and satisfies 
$\vartheta_0=\mathrm{id}$ and $\vartheta_{t+s}=\vartheta_t\circ \vartheta_s$, $s,t \in \R$.

Assume further that $\varphi:[0,\infty)\times \R^d\times \Omega \to \R^d$ is measurable and, 
for all $\omega \in \Omega$, $s,t\ge 0$,
\begin{itemize}
    \item [(i)] $\varphi(t,\omega):\R^d \to \R^d$ is continuous, 
    \item[(ii)] $\varphi(t+s,\omega)=\varphi(t,\vartheta_s\omega)\circ \varphi(s,\omega)$.
\end{itemize}
Then $\varphi$ (or $(\varphi,\vartheta)$) is called a {\em random dynamical system} (RDS).
\end{definition}

A typical example of an RDS is the solution of an SDE: if the SDE has Lipschitz coefficients and is driven by a Wiener process (or, more generally, a continuous semimartingale  with stationary 
increments), then, on a canonical space, there exists a modification of the solution 
which is a random dynamical system, see \cite{AS95}.

The concept of a random attractor of an RDS was introduced in \cite{CF94}. Later, weak random attractors were introduced by Ochs \cite{ochs99}. More recent basic papers about 
random attractors are \cite{CK15} and \cite{CS18}. In our set-up, the concept of a {\em weak} random attractor (which attracts bounded sets in probability rather than almost surely) seems appropriate. 
\begin{definition}
Let $(\varphi,\vartheta)$ be an RDS on $(\Omega,\FF,\P)$. A random set $A(\omega)$, $\omega \in \Omega$ is called a {\em (weak, random) attractor}, if
\begin{itemize}
    \item [(i)] $A$ is a compact random set.
    \item [(ii)] $A$ is strictly invariant, i.e.
    $$
    \varphi(t,\omega)A(\omega)=A(\vartheta_t\omega)
    $$
    $\P$-almost surely for every $t \ge 0$.
    \item [(iii)] $A$ attracts all bounded sets:
    $$\lim_{t \to \infty} \sup_{x \in \varphi(t,\vartheta_{-t}\omega)B}\inf_{y \in  A(\omega)}|x-y| =0\, \mbox{ in probability}$$ for every bounded set $B \subset \R^d$.
\end{itemize}

\end{definition}
The following necessary and sufficient criterion for the existence of a weak attractor is a special case of \cite[Theorem 4.2]{CDS09}.
\begin{proposition}\label{th:criterion}
The  $\R^d$-valued RDS $(\varphi,\vartheta)$ admits a weak random attractor if and only if for 
every $\varepsilon>0$ there exists some $\overline{R}>0$ such that for all $R>0$ there exists a $t_0>0$ such that for all $t \ge t_0$
$$
\P\Big(\sup_{|z|\le R}|\varphi(t,z)|  \le \overline{R} \Big)\ge 1-\varepsilon.
$$
\end{proposition}
Our aim is to show that this criterion holds for the RDS generated by \eqref{eq:driving}. 

At first, it is not clear that the equation generates an RDS on $\R^2$ for two reasons: the coefficients are not necessarily locally Lipschitz continuous in a neighborhood of 0 and 
there might be blow-up.  To deal with the first issue, we can modify the  
coefficients within the unit circle in such a way that they become Lipschitz (in Cartesian coordinates). 
This will not change the existence of an attractor. The second issue is more serious. 
Since the coefficients are locally Lipschitz continuous, equation (\ref{eq:driving}) on $\mathbb{R}^2$ generates at least a {\em local} RDS for which some trajectories might blow-up in finite time (there is no need to provide a formal definition here). In fact, we will prove the 
condition in Proposition \ref{th:criterion} for the {\em local} RDS generated by the system.  
This automatically shows that the local RDS is in fact a (global) RDS.

%

\section{Main result in the stochastic setting}
The main result of this paper states the existence of a random attractor for a certain parameter regime.
\begin{theorem}
\label{th:main}
For any $\sigma>0$ and parameter choice $w>v>1$, $\frac{2}{3}\gamma+1>v$, $w-1>\gamma$, equation  \eqref{eq:driving} generates an RDS $\varphi$. Further, $\varphi$ admits a random attractor. 
\end{theorem}

\begin{remark}
  The assumptions on the parameters $\gamma,v$ and $w$ in Theorem \ref{th:main} and in Proposition \ref{Proposition:ch3_explosion} are not comparable, i.e.~neither
  set of assumptions implies the other. There are, however, parameter values satisfying both sets of assumptions, for example $\gamma=\frac 74,v=2,w=3$.
  \end{remark}

The remainder of the paper is devoted to the proof of Theorem \ref{th:main}. We will fix the 
parameters $\gamma,v,w,\sigma$ as in the theorem. All future constants are allowed to depend on these parameters.

We split the proof into  several parts, each treating a different aspect of the problem. The first provides a quantitative estimate for the expected time to cross a neighborhood around the angle  $\pi/2$. 
There the drift in the radial component leads to explosion in the deterministic case. In the second part, we show that the drift inwards, away from this ``dangerous'' neighborhood around $\pi/2$, compensates the growth with high probability. Finally, we conclude with a Markov-like argument to extend the local existence to all times.

\section{Crossing the critical region}
\label{sec:cross_crit}
We provide an estimate on the time it takes to cross the critical region defined as $\{ (r,\phi) \colon r^w\cos^2(\phi)\leq r^{v}  \}$.
First, we introduce the one-dimensional auxiliary angular process $\tilde{\phi}^R(\psi)=\tilde{\phi}$ on $\R$ (not considered modulo $\pi$) by freezing
the radial component, i.e.
\begin{equation*}
\begin{split}
\mathrm{d}\tilde{\phi}_t &=-R^\gamma\cos^2(\tilde{\phi}_t)\mathrm{d}t+\sigma\mathrm{d}W_t \text{,}\\
\tilde{\phi}_0 &= \psi
\text{,}
\end{split}
\end{equation*}
where $R\geq 1$ is fixed and $\psi\in[0,\pi]$. Before we elucidate the relation to the angular process $\phi$, we estimate the expected time it takes to cross $\pi/2$ for $\tilde{\phi}$. Let $0\leq a< \pi/2<b <\infty$ and define
\begin{equation*}
\begin{split}
\tilde{\nu}_{a,b}(\psi)&:=\tilde{\nu}_{a,b}:=\inf\lbrace t>0\colon \tilde{\phi_t}(\psi) \notin (a,b)\rbrace,\\
\tilde{\nu}_a(\psi)    &:=\tilde{\nu}_a    :=  \lim_{b\rightarrow\infty}\tilde{\nu}_{a,b}=\inf\lbrace t>0\colon \tilde{\phi_t}(\psi)=a\rbrace
\text{,}
\end{split}
\end{equation*}
for $\psi \in [0,\pi]$. Then, for $\phi \in (a,b)$,  $u_{a,b}(\phi):=\mathbb{E}\tilde{\nu}_{a,b}(\phi)$ is given by 
\begin{equation*}
u_{a,b}(\phi)=\frac{g(b)}{f(b)} f(\phi) - g(\phi),
\end{equation*}
see  \cite[p.343]{KS91}, where
\begin{align*}
    f(\phi)&:= \int_a^\phi e^{KA(\beta)}\mathrm{d}\beta\text{,}\quad\text{and}\quad
g(\phi):= \frac{2}{\sigma^2}\int_a^\phi\int_a^\beta e^{K(A(\beta)-A(z))}\mathrm{d}z\mathrm{d}\beta
\text{,}
\end{align*}
with
\begin{equation*}
A(\phi):=\int_0^\phi \cos^2(\beta)\mathrm{d}\beta=\frac{1}{2}\phi+\frac{1}{4}\sin(2\phi)\quad\text{and}\quad
K:=\frac{2}{\sigma^2}R^\gamma\text{.}\\
\end{equation*}
Observe that the following asymptotic formula holds
\begin{align*}
\lim_{b \to \infty}u_{a,b}(\phi)
=\int_a^\phi\int_\beta^\infty e^{-K(A(z)-A(\beta))}\, \mathrm{d}z\mathrm{d}\beta=: u_{a}(\phi).
\end{align*}
The  Monotone Convergence  Theorem then implies
\[
u_{a}=\lim_{b\to\infty}u_{a,b}=\lim_{b\to\infty}\mathbb{E}\tilde{\nu}_{a,b}=\mathbb{E}\lim_{b\to\infty}\tilde{\nu}_{a,b}= \mathbb{E}\tilde{\nu}_a 
\text{.}
\]
Again due to monotonicity we have
\begin{equation}\label{monotonicity}
\mathbb{E}\sup_{\phi\in[a,\pi]}\tilde{\nu}_a(\phi)=\mathbb{E}\tilde{\nu}_a(\pi)=u_a(\pi)=\frac{2}{\sigma^2}\int_a^{\pi}\int_\beta^\infty e^{-K(A(z)-A(\beta))}\mathrm{d}z\mathrm{d}\beta
\text{.}
\end{equation}

It will turn out to be useful to know how quickly $u_a(\pi)$ converges to zero in $R$ or, equivalently, in $K$. The following lemma provides a bound on the speed of convergence. Its proof can be found in the appendix.
\begin{lemma}
\label{th:double_integral}
For each $k_0>0$, there exists a constant $C$, such that for all $K\ge k_0$
\[
\int_0^\pi\int_\beta^\infty e^{-K(A(z)-A(\beta))}\mathrm{d}z\mathrm{d}\beta \leq C K^{-\frac{2}{3}}.
\] 
\end{lemma}
\noindent

Now, fix an initial value $z=(r_0,\psi)$ with $r_0>R$ and $\psi \in [0,\pi]$. We compare the auxiliary process $\tilde{\phi}^R(\psi)$ with the angular process $\phi_t(z)$ (not interpreted modulo $\pi$). Therefore, we introduce the stopping times
\begin{equation*}
\begin{split}
\nu_a(z)&:= \inf\{t>0\colon \phi_t(z)\leq a\}\text{,}\\
\nu^R(z)&:= \inf\{t>0\colon r_t(z)\leq R\}
\text{.}
\end{split}
\end{equation*}
For $t\leq \nu^R(z)$, we have
\begin{equation*}
\begin{split}
\phi_t &=\psi -\int_0^t r_s^\gamma \cos^2(\phi_s)\mathrm{d}s +\sigma W_t\\
&\leq \psi -R^\gamma\int_0^t  \cos^2(\phi_s)\mathrm{d}s +\sigma W_t.\\
\end{split}
\end{equation*}
The following lemma yields the comparison between the auxiliary process $\tilde{\phi}$ and the angular process $\phi$.

\begin{lemma}\label{furtherlemma}
We have
\[
 \phi_t(z)\leq \tilde{\phi_t}(\psi) \qquad \text{ for all } t\leq \nu^R(z) \text{ a.s.}.
\]
In particular,
\[
\nu_a(z)\wedge\nu^R(z) \leq \tilde{\nu}_a(\psi)
\text{.}
\]
\end{lemma}

\begin{proof}
  Let $g_t:=\phi_t - \tilde{\phi}_t$ and let $L>0$ be a Lipschitz constant of the function $y \mapsto \cos^2(y)$. Then
$$
g'_t \leq R^\gamma L |g_t|,\;t \in [0,\nu^R(z)].
$$
Since $g_0=0$, Gronwall's inequality implies $g_t\le 0$ on the interval $[0,\nu^R(z)]$, so the statement of the lemma follows.
\end{proof}
\noindent
The following lemma follows from \eqref{monotonicity}, Lemma \ref{th:double_integral}, and Lemma \ref{furtherlemma}.
\begin{lemma}
\label{lemma:timebound}
There exists a constant $C=C(\sigma)$ such that for all $r_0 > R\geq 1$, $a\in[0,\pi/2)$ we have
\begin{equation*}
    \mathbb{E}\left[\sup_{z\in\{r_0\}\times [0,\pi]} \nu_a(z)\wedge\nu^R(z)\right] \leq C R^{-\frac{2}{3}\gamma}
\text{.}
\end{equation*}
\end{lemma}

\subsection{Step down}
\label{sec:step_down}
Unfortunately, crossing the critical region quickly is not enough, since trajectories will reenter inevitably.
Therefore, we have to use the strong negative drift $-r_t^w\cos^2(\phi_t)$ outside the critical region to compensate the growth.

We fix $1>\varepsilon>\tilde{\varepsilon}>0$ (small), $\sigma>0, T\geq 1, d>0$. For $k\in\mathbb{N}_0$, we set $R_k:= 2^k$ and define a sequence of regions which will be useful showing how the radius decreases outside of the critical region. Namely, we define
\begin{align*}
\mathsf{B}_k:=[R_k,\infty) \times[\pi/2-\varepsilon,\pi)\quad\text{and}\quad I(k) := \{R_k\}\times [0,\pi]\text{.}
\end{align*}
Next we define the following sequences of stopping times to track the dynamics movement through the regions.
\begin{align*}
\overline{\tau}_{k+1}&:= \inf\left\lbrace t\geq 0\colon \sup_{z\in I(k)} r_t(z) \geq R_{k+1} \right\rbrace\text{,}\\ 
\underline{\tau}_{k-1}&:= \inf\left\lbrace t\geq 0\colon \sup_{z\in I(k)} r_t(z) \leq R_{k-1} \right\rbrace\text{,}\;k \geq1,\\
\underline{\tau}_k(z)&:= \inf\left\lbrace t\geq 0\colon r_t(z) \leq R_k \right\rbrace\text{,}\\
\tau_k(z) &:= \inf\left\lbrace t\geq 0\colon (r_t(z),\phi_t(z)) \notin \mathsf{B}_k \right\rbrace=\nu_{\pi/2-\varepsilon}(z)\wedge\nu^{R_k}(z)\text{.}
\end{align*}
Fixing the positive constant $\tilde{c}:= \cos^2\left(\frac{\pi}{2}-\varepsilon+\tilde{\varepsilon}\right)$, we next define a collection of events which will be used to control the stopping times just defined.
\begin{align*}
\overline{\text{BBM}}_k&:=\left\lbrace \sup_{0\leq s\leq t\leq T}\sigma(W_t-W_s)-\frac{\tilde{c}}{2}R_{k-2}^\gamma (t-s)\leq \tilde{\varepsilon}/2\right\rbrace\text{,}\;k\geq 2, \\
\underline{\text{BBM}}_k&:=\left\lbrace \inf_{0\leq s\leq t\leq T}\sigma(W_t-W_s)+R_{k+1}^\gamma (t-s)\geq -d\right\rbrace \text{,}\\
\text{BBM}_k&:= \overline{\text{BBM}}_k\cap\underline{\text{BBM}}_k\text{.}
\end{align*}
In the following, we estimate the probability of $\lbrace \underline{\tau}_{k-1}\geq \overline{\tau}_{k+1}\rbrace$ (for large $k$). We begin with a few observations which will help illuminate the structure of the objects we just defined.

The constant $\tilde{c}$ will serve as a lower bound for $\cos^2(\phi_t)$ outside of the critical region. We will refer to $\mathcal{G}:=\{(\phi,r)\colon \cos^2(\phi)\geq \tilde{c},r\geq 1\}$ as the {\em good region}. Inside the good region the drift in the $r$-component is negative (at least for large $r$) and furthermore this region is insensitive to noise.
The event $\overline{\text{BBM}}_k$ will guarantee that one-point motions which leave the critical region will not directly reenter, whereas $\underline{\text{BBM}}_k$ ensures that the trajectories do not move through the good region too quickly, see proof of Proposition \ref{th:core}, in particular the bounds on $\phi_{t\wedge\vartheta}(z)$.

\subsubsection{Velocity in the radial direction}\label{veloradial}
In the following we derive an estimate for the minimal time it takes to go from level $r\ge 1$ to a higher level $R>r$. Because of
\[
\mathrm{d}r_t =\left(-r_t^w\cos^2(\phi_t)+r^v_t\right)\mathrm{d}t \leq r^v_t\mathrm{d}t
\]
one can conclude, via a comparison argument, that the process $(\tilde{r}_t)_{t\geq 0}$ solving the (deterministic) equation
\[
\frac{\dd}{\dd t}\tilde{r}_t =\tilde{r}_t^v, \qquad \tilde{r}_0=r
\]
is an upper bound for the radial component $(r_t)_{t\geq 0}$ starting anywhere in $[0,r]$. This ordinary one-dimensional differential equation can be solved explicitly
\[
\tilde{r}_t =\frac{r}{\left(1-(v-1)r^{v-1}t \right)^{\frac{1}{v-1}}}.
\]
Hence, 
\[
d(r,R):= \frac{1}{v-1}\left(\frac{1}{r^{v-1}}-\frac{1}{R^{v-1}}\right)\text{,} \qquad R\geq r\ge 1
\]
is a lower bound for the time the radial component $(r_t)_{t\geq 0}$ needs to go from $r$ to $R$.

As long as  a trajectory stays in the good region and its radial component is at least 
$\rho_0:=(\tilde c/2)^{1/(v-w)}$ we have
\[
\mathrm{d}r_t =\left(-r_t^w\cos^2(\phi_t)+r^v_t\right)\mathrm{d}t\leq \left(-\tilde{c}r_t^w+r^v_t\right)\mathrm{d}t  \leq -\frac{\tilde{c}}{2}r_t^w\mathrm{d}t
\]
and therefore, the process $(\hat{r}_t)_{t\geq 0}$ solving the (deterministic) equation
\[
\frac{\dd}{\dd t}\hat{r}_t =-\frac{\tilde{c}}{2}\hat{r}_t^w, \qquad \hat{r}_0=R
\]
is an upper bound for the radial component $(r_t)_{t\geq 0}$ starting anywhere in $[0,R]$ as
long as $\hat{r}_t \geq \rho_0$. 
Solving the ODE, we get
\[
\hat{r}_t =\frac{R}{\left(1+(w-1)\frac{\tilde{c}}{2} R^{w-1}t \right)^{\frac{1}{w-1}}}.
\]
Therefore,
\[
d_{\frac{\tilde{c}}{2},w}(R,r):= \frac{2}{\tilde{c}(w-1)}\left(\frac{1}{r^{w-1}}-\frac{1}{R^{w-1}}\right)\text{,} \qquad R\geq r\geq\rho_0
\]
is an upper bound for the time the radial component $(r_t)_{t\geq 0}$ starting somewhere in $(r,R]$ needs go to $r$ inside the good region $\mathcal{G}$.

\subsubsection{Proof of the step down}
We now define the following events in the interest of brevity (recalling the notation from the start of Section~\ref{sec:step_down}):
\begin{multline*}
        A_k := \left\lbrace\underline{\tau}_{k-1}\leq \theta_k \leq \overline{\tau}_{k+1}\right\rbrace   \text{,}\quad
B_k := \left\lbrace\sup_{z\in I(k)} \tau_{k-2}(z)\leq d\left(R_k,R_{k+1}\right)\right\rbrace\\
\text{and}\quad D_k := \left\lbrace\sup_{z\in I(k)} \underline{\tau}_{k-2}(z) > \theta_k\right\rbrace
\end{multline*}
where  $\theta_k =d\left(R_k,R_{k+1}\right)+d_{\tilde{c}/2,w}(R_{k+1},R_{k-2})$, $k \ge 2$. Note that 
there exists some $\check c>0$ such that
\begin{equation}\label{eq:theta}
\theta_k\sim\check c2^{-k(v-1)},\; k \to \infty.
\end{equation}

\begin{figure}[H] 






\includegraphics{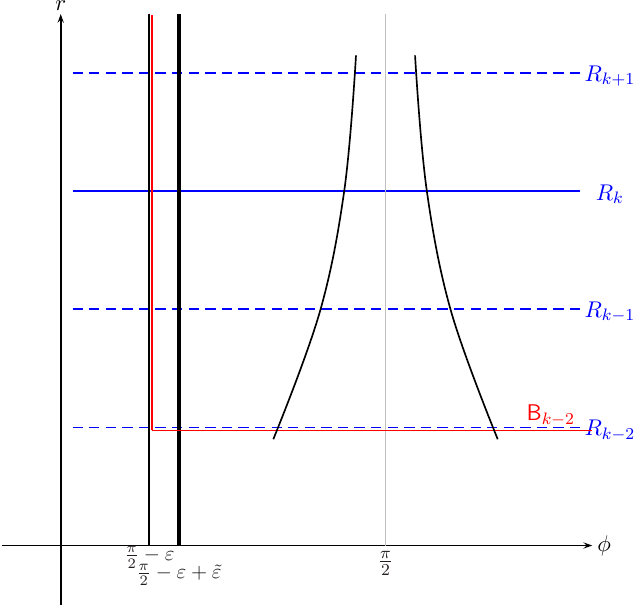}
\caption{Step Down}
\label{fig:stepdown}
\end{figure}



We begin by showing that the probability of $A_k$ converges to $1$ as $k$ tends to infinity. On $A_k$ all points that start at the level $R_k$ are below $R_{k-1}$ before time $\theta_k$ while none ever exceeded the level $R_{k+1}$ before. We call this ``\textit{step down}''.
We estimate the probability of $A_k$ in terms of the probability of $B_k$ and $\text{BBM}_k$, of which we can compute explicit bounds. $D_k$ is just an auxiliary event, which cannot occur at the same time as $B_k$ and $\text{BBM}_k$, as we show below. 

On the event $\{\sup_{z \in I(k)} \tau_{k-2}(z) \leq d(R_k,R_{k+1})\}$ all one-point motions starting in $I(k)$ leave the unbounded box $\mathsf{B}_{k-2}$ 
(see Figure \ref{fig:stepdown}) 
before they can pass $R_{k+1}$, since the time to do so is at least $d(R_k,R_{k+1})$. Leaving this box means being below $R_{k-2}$ or being strictly inside the good region. We will show that if, in addition, $\text{BBM}_k$ holds, then  trajectories in the good region will necessarily go under $R_{k-2}$. Afterwards there is not enough time to go above $R_{k-1}$ again, hence all points are simultaneously below $R_{k-1}$ at or before time $\theta_k$.
\begin{proposition}
\label{th:core}
For sufficiently large $k \ge 2$, the following hold true
\begin{itemize}
\item[(i)] $B_k \cap \mathrm{BBM}_k \cap D_k =\emptyset$,
\item[(ii)] $B_k \cap \mathrm{BBM}_k =B_k \cap \mathrm{BBM}_k\cap D_k^\mathrm{c} \subset A_k$,
\item[(iii)] $\mathbb{P}\left(B_k^\mathrm{c}\right),\mathbb{P}\left(\mathrm{BBM}_k^\mathrm{c}\right)\rightarrow 0$ as $k\to \infty$.
\end{itemize}
\end{proposition}

\begin{proof}[Proof of Proposition~\ref{th:core}]
  We assume that $k$ is so large that $R_{k-2}\ge \rho_0$ (defined in Subsection \ref{veloradial}) and that  $\theta_k \leq T$. (Recall that $T\ge 1$ 
  was fixed at the start of Section~\ref{sec:step_down}.) \\
  
\textit{Proof of claim (i):}
On 
\begin{align*}
\lbrace\sup_{z\in I(k)} \tau_{k-2}(z) \leq d(R_k,R_{k+1})\} \cap\text{BBM}_k \cap \{\sup_{z\in I(k)} \underline{\tau}_{k-2}(z) > \theta_k\rbrace\,,    
\end{align*}
there is a $z\in I(k)$, such that
\begin{equation}\label{tau}
    \tau_{k-2}(z) \leq d(R_k,R_{k+1}) \quad \text{and}\quad \underline{\tau}_{k-2}(z) \stackrel{}{>} \theta_k
\text{,}
\end{equation}
which implies
\[
\phi_{\tau_{k-2}(z)}(z)=\frac{\pi}{2}-\varepsilon \quad \text{and}\quad r_{\tau_{k-2}(z)}(z) \in [R_{k-2},R_{k+1}]
\text{.}
\]
Now we show that on the set $\text{BBM}_k$ the process $(\phi(z),r(z))$ will spend sufficient time in the good region $\mathcal{G}$ that it will necessarily go below $R_{k-2}$ before time $\theta_k$. Since this is a contradiction to $\underline{\tau}_{k-2}(z) > \theta_k$, $B_k \cap \text{BBM}_k \cap D_k$ is empty.

Define $\vartheta\coloneqq\inf\{t\geq \tau_{k-2}(z)\colon \cos^2(\phi_t(z)) \leq \tilde{c}\}\wedge \underline{\tau}_{k-2}(z)$. Let $t \geq \tau_{k-2}(z)$, then, using the fact that $\theta_k \leq T$,  we can conclude
\begin{equation*}
\begin{split}
\phi_{t\wedge\vartheta}(z) &= \phi_{\tau_{k-2}(z)}(z)-\int_{\tau_{k-2}(z)}^{t\wedge\vartheta} (r_s(z))^\gamma\cos^2(\phi_s(z))\mathrm{ds} +\sigma\left(W_{t\wedge\vartheta}-W_{\tau_{k-2}(z)}\right)\\
&\geq \frac{\pi}{2}-\varepsilon-R_{k+1}^\gamma(t\wedge\vartheta-\tau_{k-2}(z))+\sigma\left(W_{t\wedge\vartheta}-W_{\tau_{k-2}(z)}\right)\\
&\geq \frac{\pi}{2}-\varepsilon-d-2R_{k+1}^\gamma(t\wedge\vartheta-\tau_{k-2}(z))\\
\text{and} \hspace{20pt} &\\
\phi_{t\wedge\vartheta}(z) &\leq \frac{\pi}{2}-\varepsilon -R_{k-2}^\gamma\tilde{c}(t\wedge\vartheta-\tau_{k-2}(z))+\sigma\left(W_{t\wedge\vartheta}-W_{\tau_{k-2}(z)}\right)\\
&\leq \frac{\pi}{2}-\varepsilon+\frac{\tilde{\varepsilon}}2 -R_{k-2}^\gamma\frac{\tilde{c}}{2}\left(t\wedge\vartheta-\tau_{k-2}(z)\right)< \frac{\pi}2 -\varepsilon + \tilde{\varepsilon}
\text{,}
\end{split}
\end{equation*}
where, for a moment, we regard the process $\phi$ as $\R$-valued rather than $[0,\pi]$-valued.

The upper bound shows that the process $\phi(z)$ does not exit the interval $[-\pi/2+\varepsilon-\tilde\varepsilon,\pi/2 -\varepsilon+\tilde \varepsilon]$ via 
the right end point up to time $\vartheta$ and the lower bound shows that 
the process cannot hit the left end point before the minimum of $\underline{\tau}_{k-2}(z)$
and 
\[
t_0=\tau_{k-2}(z) + \underbrace{\frac{\pi-2\varepsilon+\tilde{\varepsilon}-d}{2R_{k+1}^\gamma}}_{\text{\parbox[t][][t]{12em}{\small least amount of time\\ spent in good region}}}\text{.}
\]
We know from the previous subsection that the time it takes in the good region to go from $R_{k+1}$ down to $R_{k-2}$ is at most
\[
d_{\frac{\tilde{c}}{2},w}(R_{k+1},R_{k-2})\coloneqq \frac{2}{\tilde{c}(w-1)}\left(\frac{1}{R_{k-2}^{w-1}}-\frac{1}{R_{k+1}^{w-1}}\right)
\]
for $k$ large enough which is smaller than $t_0-\tau_{k-2}(z)$ for $k$ large enough
since $w-1>\gamma$. 
This implies 
\begin{multline*}
\underline{\tau}_{k-2}(z)\leq \tau_{k-2}(z)+d_{\frac{\tilde{c}}{2},w}(R_{k+1},R_{k-2})\\ 
\leq d\left(R_k,R_{k+1}\right)+d_{\frac{\tilde{c}}{2},w}(R_{k+1},R_{k-2}) = \theta_k
\text{,}    
\end{multline*}
which contradicts the second inequality in  \eqref{tau}.\\

\textit{Proof of claim (ii):}
The first equality is just a reformulation of claim (i).
On the event $\{ \sup_{z\in I(k)} \tau_{k-2}(z) \leq d(R_k,R_{k+1})\} \cap\text{BBM}_k \cap \{\sup_{z\in I(k)} \underline{\tau}_{k-2}(z) \leq \theta_k \}$ each trajectory starting in $I(k)$ hits level $R_{k-2}$ before time $\theta_k$. From the calculation in the proof of claim (i) we  know that none of these exceeds $R_{k+1}$ until $\theta_k$. On the other hand, the time it takes for any trajectory to go from $R_k$ to $R_{k-2}$ and then above $R_{k-1}$ is at least
\[
d_{1,w}(R_k,R_{k-2})+d(R_{k-2},R_{k-1}) \eqqcolon \tilde{\theta}_k.
\]
We claim that, for $k$ large enough, we have
\begin{align}\label{multi}
\begin{split}
\theta_k=d(R_k,R_{k+1})+&d_{\tilde{c}/2,w}(R_{k+1},R_{k-2})\\ \leq &d_{1,w}(R_k,R_{k-2})+d(R_{k-2},R_{k-1})=\tilde{\theta}_k
\text{.}    
\end{split}
\end{align}
Note that
\begin{align*}
d(R_{k-2},R_{k-1})- d(R_k,R_{k+1}) &= \frac{1}{v-1}\left( R_{k-2}^{1-v}-R_{k-1}^{1-v}-R_{k}^{1-v}+R_{k+1}^{1-v} \right)\\
&= \frac{1}{v-1}R_{k}^{1-v}\left( 4^{v-1}-2^{v-1}-1+2^{1-v} \right)\\
&=: \beta_v R_k^{1-v}.
\end{align*}
The function $x\mapsto 4^x-2^x-1+2^{-x}$ is strictly increasing on $[0,\infty)$ and $0$ at $0$. Since $v>1$ we get $\beta_v>0$. Note that, since $w>v$, the two remaining terms in 
\eqref{multi} decay faster than $R_k^{1-v}$ as $k \to \infty$ and therefore \eqref{multi}
holds true for all sufficiently large $k$. We conclude that at time $\theta_k$ all 
trajectories starting in $I(k)$ are below level $R_{k-1}$ and never hit level $R_{k+1}$ 
up to that time.\\

%

\textit{Proof of claim (iii):}
Note that $\tau_k(z)$ coincides with $\nu_{\pi/2-\varepsilon}(z)\wedge\nu^{R_k}(z)$ defined in Section \ref{sec:cross_crit}. Applying Chebyshev's inequality yields
\begin{align*}
\mathbb{P}\left(B_k^\mathrm{c}\right)&=\mathbb{P}\left(\sup_{z\in I(k)} \tau_{k-2}(z) > d(R_k,R_{k+1})\right)
\leq \frac{\mathbb{E}\left(\sup_{z\in I(k)} \nu_{\pi/2-\varepsilon}(z)\wedge\nu^{R_{k-2}}(z)\right)}{d(R_k,R_{k+1})}\,.
\end{align*}
Combining this estimate with Lemma~\ref{lemma:timebound} yields, for $k$ large, 
\begin{align*}
\mathbb{P}\left(B_k^\mathrm{c}\right)&\leq (v-1)C(\sigma) \frac{R_{k-2}^{-\frac{2}{3}\gamma}}{\dfrac{1}{R_{k}^{v-1}}-\dfrac{1
}{R_{k+1}^{v-1}}}\\
&=(v-1)C(\sigma) \frac{1}{4^{1-v}-8^{1-v}} R_{k-2}^{-\frac{2}{3}\gamma+v-1}= \textcolor{black}{\tilde{C}(v,\gamma,\sigma)R_k^{-\frac{2}{3}\gamma+v-1}}
\text{.}
\end{align*}
This bound tends to $0$ as $k\to\infty$ because of $v<2\gamma/3+1$.

To estimate $\mathbb{P}\left(\text{BBM}_k^\mathrm{c}\right)$, we need the following lemma which we prove in the appendix.
\begin{lemma}\label{le:downfall}
For every $\sigma>0$, $r \ge \sqrt{2}\sigma$, and  $\varepsilon>0$, we have
\[
\mathbb{P}\left(\sup_{0\leq s\leq t\leq T} \left(\sigma (W_t-W_s)-r(t-s)\right) > \varepsilon\right)\leq  \frac{4\sqrt{2}}{\sigma^2}\mathrm{e}^{T} r^2 \mathrm{e}^{-2\frac{\varepsilon}{\sigma^2} r}.
\]
\end{lemma}

\noindent This lemma immediately implies $\mathbb{P}\left(\overline{\text{BBM}}_k^\mathrm{c}\right) \rightarrow 0$ as $k\rightarrow \infty$.  Furthermore,
\[
\mathbb{P}\left(\underline{\text{BBM}}_k\right)=\mathbb{P}\left(\sup_{0\leq s\leq t\leq T}\left(\sigma(W_s-W_t)-R_{k+1}^\gamma (t-s)\right)\leq d\right)
\text{,}
\]
and therefore $\mathbb{P}(\text{BBM}_k^\mathrm{c})\leq \mathbb{P}\left(\overline{\text{BBM}}_k^\mathrm{c}\right)+\mathbb{P}\left(\underline{\text{BBM}}_k^\mathrm{c}\right)$ tends to 0 as $k \to \infty$.
\end{proof}
\noindent
The following corollary is a consequence of Proposition \ref{th:core} (ii) and (the proof of) (iii).
\begin{corollary}
\label{co:vanishing_probability}
There exists a constant $C$ such that 
\[
\mathbb{P}\left(A_k^c\right) \leq C \exp\Big\{-k\Big(\frac 23\gamma-v+1\Big)\log2\Big\}.
\]
for all $k \ge 2$. In particular, we have $\lim_{k \to \infty}\mathbb{P} \big(A_k\big)=1.$
\end{corollary}

\begin{proof}
Since 
$$
\mathbb{P}\big(A_k^c\big)\le \mathbb{P}\big(B_k^c\big)+\mathbb{P} \big(\mathrm{BBM}_k^c\big)
$$
the estimate follows from the estimates in the proof of claim (iii) of the previous proposition.
\end{proof}

\section{Strong completeness and existence of an attractor}
\label{sec:markov argument}
It is almost clear from (the last part of) Corollary \ref{co:vanishing_probability} that under the assumptions of the main theorem, the SDE \eqref{eq:driving} is strongly complete or generates an RDS, 
i.e.~almost surely, {\em all} trajectories are global, for the following reason: if we start with a ball of radius $2^k$ around the origin, then, as we showed, with high probability its image under the (local) 
flow after a short time will be contained in a ball of radius $2^{k-1}$ and not a single point in the 
original ball will reach distance $2^{k+1}$ from the origin within this time. Iterating this argument 
indicates that the image of a bounded set under the (local) RDS cannot blow-up.

\subsection{Dominating jump process}
We will now define a piece-wise constant process $K_t$, $t \ge 0$ taking values in $\{2,3,...\}$ which, in some sense, dominates the radius of the image of a bounded subset of $\R^2$ under the RDS generated by \eqref{eq:driving}. Recall from Section~\ref{sec:step_down} the definitions $R_k:=2^k$, $I(k)= \{ R_k\}\times [0,\pi]$, $k \in \N_0$  and $\theta_k$ for $k\geq 2$. 
We define the process $(K_t)_{t\geq 0}$ starting in  $k\geq 2$ and 
associated stopping times in a recursive way as follows:

$\tau^0 \coloneqq 0$, $ K_0\coloneqq k$, $\tau^{n}\coloneqq \underline{\tau}^{n} \wedge \overline{\tau}^{n}$, $K_t=k$ for $t \in [0,\tau^1)$  and
\begin{equation*}
\begin{split}
K_t=
\begin{cases}
K_{\tau^{n-1}}-1 & \text{if } t\in[\tau^n,\tau^{n+1})\text{ and  }\underline{\tau}^n\leq\overline{\tau}^n \text{,} \\
K_{\tau^{n-1}}+1 & \text{if } t\in[\tau^n,\tau^{n+1})\text{ and  }\underline{\tau}^n>\overline{\tau}^n\\
\end{cases}
\end{split}
\end{equation*}
for $n \ge 1$, where
\begin{equation*}
\begin{split}
\underline{\tau}^{n}&\coloneqq \inf \left\lbrace t> \tau^{n-1}\colon 2 < \sup_{z\in I(K_{\tau^{n-1}})} r_{\tau^{n-1},t}(z) \leq R_{K_{\tau^{n-1}}-1}\right\rbrace \text{,} \\
\overline{\tau}^{n}&\coloneqq \inf \left\lbrace t> \tau^{n-1}\colon \sup_{z\in I(K_{\tau^{n-1}})} r_{\tau^{n-1},t}(z) \geq R_{K_{\tau^{n-1}}+1}\right\rbrace\wedge \left(\theta_{K_{\tau^{n-1}}}+\tau^{n-1}\right) \text{,} \\
\end{split}
\end{equation*}
where we denote the radial part of the solution at $t$ starting in $z$ at time $s\leq t$ 
by $r_{s,t}(z)$.
Note that $\underline{\tau}^n=\infty$ if $K_{\tau^{n-1}}=2$, so $K$ cannot jump to values smaller than 2 
(which would cause problems since $\theta_1$ is not defined).

We will see that $\tau:=\lim_{n\to\infty} \tau^{n}=\infty$ almost surely, so that $K_t$ is 
 well-defined for all $t \ge 0$.

Taking the minimum with $\theta_{K_{\tau^{n-1}}}+\tau^{n-1}$ in the definition of $\overline{\tau}^n$ yields a uniform upper bound for $\tau^n-\tau^{n-1}$, namely $\theta:=\sup_{k \ge 2}\theta_k$.

The following lemma follows easily from the definitions by induction on $n \in \N_0$.
\begin{lemma}
\label{lemma:dominator}
For any initial value $k\geq 2$ the following holds:
\begin{itemize}
\item[a)]
\[
\sup_{z\in I(k)} r_{\tau^{n}}(z) \leq R_{K_{\tau^n}}
\quad \text{for all } n\in\mathbb{N}_0,
\] 
\item[b)] 
\[
\sup_{z\in I(k)} r_{t}(z) \leq R_{K_{t}+1}
\quad\text{for all } t< \tau \text{.}
\]
\end{itemize}
\end{lemma}

\subsection{Semi-Markov property}
\label{sec:semimarkov}
In this subsection we recall parts of the theory of Semi-Markov processes and Markov renewal processes. For a more detailed description we refer to \cite[p. 313ff]{cinlar75}.

Let, for each $n\in\mathbb{N}_0$, $X_n$ be a random variable taking values in a countable set $E$ and $T_n$ an $\mathbb{R}_+$-valued random variable, such that $0=T_0< T_1<T_2 < \dots $.
\begin{definition}
We call the stochastic process $(X_n,T_n)_{n\in\mathbb{N}_0}$\textit{ Markov renewal process with state space $E$}, if
\begin{equation*}
\begin{split}
&\mathbb{P}\left(X_{n+1}=j, T_{n+1}-T_n\leq t \vert X_0,\dots,X_n;T_0,\dots T_n \right)\\
& \hspace{0pt} = \mathbb{P}\left(X_{n+1}=j, T_{n+1}-T_n\leq t \vert X_n\right)
\end{split}
\end{equation*}
holds for all $n\in\mathbb{N}$, $j\in E$ and $t\in\mathbb{R}_+$.
The process $Z_t= X_n$ for $t\in [T_n,T_{n+1})$ is called a \textit{Semi-Markov process} and $(X_n)_{n\in\mathbb{N}_0}$ its \textit{embedded Markov chain}.
\end{definition}
\begin{remark}
A Markov renewal process is called \textit{time-homogeneous}, if in addition
\[
\mathbb{P}\left(X_{n+1}=j, T_{n+1}-T_n\leq t \vert X_n=i\right)=\mathbb{P}\left(X_{1}=j, T_{1}\leq t \vert X_0=i\right)
\]
holds for any $n\in\mathbb{N}$, $i,j \in E$ and $t\in\mathbb{R}_+$.
\end{remark}

In the following we only work with time-homogeneous processes.

\noindent
We denote for $i,j \in E$
\begin{equation*}
\begin{split}
\mathbb{P}_i & := \mathbb{P}\left( \cdot \vert X_0=i\right)\text{,}\\
m(i) &:= \mathbb{E}_i T_1\text{, where }\mathbb{E}_i  \text{ is the expectation under } \mathbb{P}_i\text{,}\\
P_{i,j}& \mathbb{P}_i\left(X_{1}=j\right) = \lim_{t\to\infty}\mathbb{P}_i\left(X_{1}=j, T_{1}\leq t\right)\text{,}\\
P&:= \left( P_{i,j}\right)_{i,j\in E}\text{.}
\end{split}
\end{equation*}
\begin{definition}
A Semi-Markov process is called \textit{irreducible} resp.~\textit{recurrent} if the corresponding embedded Markov chain is irreducible resp.~recurrent.
Further, let $S_1^{j}, S_2^j, \dots$ be a sub-sequence of $T_1, \dots$, such that $S_1^{j} < S_2^{j} <\cdots$ and $Z_{S_n^j}=j$ for every $n\in\mathbb{N}$. $(Z_t)_{t\geq 0}$ is called \textit{periodic with period $\delta$} if $S_1^j, S_2^j-S_1^j, S_3^j-S_2^j, \dots$ take values in a discrete set $\{0, \delta, 2\delta, \dots\}$ where $\delta>0$ is the largest such number. If there is no such $\delta>0$,  $(Z_t)_{t\geq 0}$ is called \textit{aperiodic}.
\end{definition}
\begin{theorem}
\label{th:semi_covergence}
If the Semi-Markov process is irreducible, recurrent and  aperiodic, $\nu$ is a non-trivial 
non-negative solution to $\nu=\nu P$ and $m(k)<\infty$ for all $k\in E$, then for any $i\in E$
\[
\lim_{t\to\infty} \mathbb{P}_i\left(Z_t=j\right)=\frac{1}{\nu m}\nu(j)m(j)
\text{,}
\]
provided that $\nu m \coloneqq \sum_{j\in E}\nu(j)m(j) < \infty$.
\end{theorem}
\noindent
The proof can be found in \cite[p. 342]{cinlar75}.

Now we apply these general results to our set-up. 
We define $T_n \coloneqq \tau^{n}$ and $X_n \coloneqq K_{\tau^n}$. By definition, $(X,T)$ is a Markov renewal process with state space $\{2,3,...\}$ and $(K_t)_{t\geq 0}$ is the corresponding Semi-Markov process. Further, the embedded Markov chain $(X_n)_{n\in\mathbb{N}}$ has the transition probabilities
\begin{equation*}
P_{i,j}=\begin{cases}
1-p_i & \text{if } j=(i-1 ) \vee 2, i \geq 2 \text{,} \\
p_i   & \text{if } j=i+1, i\geq 2         \text{,} \\
0     & \text{else,}
\end{cases}
\end{equation*}
where $p_i\coloneqq \mathbb{P}_i(X_1=i+1)$, $i\geq 2$. Note that $p_2=1$ and
 $p_i>0$, $i \geq 2$. We will not investigate whether $p_i < 1$ for all $i \geq 3$. This property certainly holds for large enough $i$ and we will pretend that it holds for 
 all $i \geq 3$. If not, then we can consider a subset $\{k_0,...\}$ on which this property holds. In any case, $(X_n)_{n\in \mathbb{N}}$ is irreducible and therefore also $(K_t)_{t\geq 0}$.
\begin{lemma}
\label{th:vanishing_probability}
For $w>v>1$, $\frac{2}{3}\gamma+1>v$, $w-1>\gamma$ it holds for each $\sigma>0$
\[
p_i \leq 1 -\mathbb{P}\left(A_i\right) \to 0 \text{ as } i \to \infty
\text{.}
\]
\end{lemma}
\begin{proof}
Note the difference between the stopping times $\underline{\tau}^i,\overline{\tau}^{i}$ from the definition of the dominating jump process $(K_t)_{t\geq 0}$ and $\underline{\tau}_i,\overline{\tau}_{i}$ from Section \ref{sec:step_down}.
Take $k=i$ as the initial value for the jump process $(K_t)_{t\geq 0}$.
\begin{equation*}
\begin{split}
p_i &= \mathbb{P}_i(X_1=i+1) = \mathbb{P}_i\left( \underline{\tau}^1>\overline{\tau}^1  \right)\\
& = 1- \mathbb{P}\left( \underline{\tau}_{i-1}\leq \overline{\tau}_{i+1}\wedge \theta_i  \right)\\
&\leq 1- \mathbb{P}\left(A_i\right) \underset{\text{Cor.}~\ref{co:vanishing_probability}}{\longrightarrow} 0 \text{ as } i\to \infty
\text{.}
\end{split}
\end{equation*}
\end{proof}

The Markov chain $(X_n)_{n\in \mathbb{N}}$ is recurrent and has an invariant probability distribution $\nu$ if
\[
\sum_{i=3}^\infty \prod_{j=2}^{i-1} \frac{p_j}{1-p_{j+1}} < \infty
\text{.}
\]
This series is indeed finite by Cauchy's ratio test and Lemma \ref{th:vanishing_probability}. Further,  $\tau=\infty$ a.s., see \cite[Prop. 3.16, p. 327]{cinlar75}. This implies that the (local) flow $\left(r_{s,t}(z),\phi_{s,t}(z)\right)$ is in fact global.

Finally, recall that $T_1=\tau^1$ and $\tau^1 \leq \theta_i \hspace{8pt} \mathbb{P}_i\text{-a.s.}$, which yields 
\[
m(i)=\mathbb{E}_i T_1 \leq \theta_i\leq \theta
\text{.}
\]
Hence, $\nu m= \sum_{i=2}^\infty m(i)\nu(i)\leq \theta< \infty$ and therefore the measure $\mu$ defined via $\mu(i) = \frac{1}{\nu m} \nu(i) m(i) $ is a probability measure.

Now we only need to ensure that the Semi-Markov process is aperiodic in order to be able to apply Theorem \ref{th:semi_covergence}. For any $\delta>0$ there is a $n\in\mathbb{N}$ such that $\theta_n<\delta$. Hence, the period of $n$ must be less than $\delta$. But since periodicity is a class property the Semi-Markov process $(K_t)_{t\geq 0}$ is aperiodic.

\subsection{Existence of an attractor}
Now, we prove the existence of a random attractor. Recall the equivalent criterion from Proposition \ref{th:criterion}.  
\subsubsection{Proof of Theorem \ref{th:main}}
Let $\varepsilon>0$, and $k\ge 2$, and $\overline{R}>0$. Lemma \ref{lemma:dominator} yields
\begin{equation}
\label{eq:proof1}
\mathbb{P}\left( \sup_{z\in [0,R_k]\times[0,\pi]}r_t(z)\leq \overline{R}\right) \geq \mathbb{P}_k\left( R_{K_t+1}\leq \overline{R}\right) =   \mathbb{P}_k\left( K_t\leq \log_2\overline{R} -1\right).
\end{equation}
Pick $\overline{R}=\overline{R}(\varepsilon)>0$, such that 
\begin{equation}
\label{eq:proof3}
\mu\left( \lbrace 2,\,3,\dots, \lfloor \log_2\left(\overline{R}\right)\rfloor -1\rbrace \right)  \geq 1- \frac{\varepsilon}{2} 
\text{.}
\end{equation}
Theorem \ref{th:semi_covergence} says that for each $A\subset \{2,3,...\}$ there is a $t_0(\varepsilon,k)=t_0>0$, such that for all $t\geq t_0$
\begin{equation}
\label{eq:proof2}
\vert \mathbb{P}_k\left(K_t \in A \right) -\mu(A) \vert \leq \frac{\varepsilon}{2} 
\text{.}
\end{equation}
Plugging (\ref{eq:proof2}) and (\ref{eq:proof3}) into (\ref{eq:proof1}) yields
\begin{equation*}
\begin{split}
\mathbb{P}\left( \sup_{z\in [0,R_k]\times[0,\pi]}r_t(z)\leq \overline{R}\right) &\geq  \mathbb{P}_k\left( K_t\leq \log_2\overline{R}-1\right)\\
&\geq \mu\left( \lbrace 2,\dots, \lfloor \log_2\overline{R}\rfloor -1\rbrace \right) -\frac{\varepsilon}{2}\\
&\geq 1-\varepsilon
\text{}
\end{split}
\end{equation*}
for sufficiently large $t$. Therefore, $\varphi$ admits a weak random attractor by Proposition \ref{th:criterion}.

\subsection{Tail estimate}
In this section we  estimate the tails of the diameter of the attractor and of the invariant measure.
Thanks to the previous analysis we have an upper bound for the radius of the attractor $A(\omega)$ in terms of the invariant probability measure $\mu$ of the process $K$.
\begin{corollary}
For parameters $w>v>1$, $\frac{2}{3}\gamma+1>v$, $w-1>\gamma$ and for each $\sigma>0$ there exists some $c$ such that, for all $x\geq 1$,  
\begin{equation*}
\mathbb{P}\left(  \sup_{z \in A(\omega)}|z|\ge x\right) \leq  \sum_{k= \lceil \log_2(x)-1\rceil \vee 2} ^\infty \mu\left(k  \right)\leq 2^{-\frac{1}{2}(\lfloor \log_2(x)-1\rfloor)^2\left(1+\frac{2}{3}\gamma-v\right)+c(\log_2x +1)}
\text{.}
\end{equation*}
In particular, all (polynomial) moments of $\sup_{z \in A(\omega)}|z|$ are finite.
\end{corollary}

\begin{proof}
We start with the first inequality. Since $A(\omega)$ is almost surely compact, it follows that for 
given $\varepsilon>0$ there exists some $k \in \N$, $k \ge 2$ such that $A(\omega)$ is contained  
inside the ball of radius $R_k$ with center 0 with probability at least $\varepsilon$. 
By invariance of $A$ we obtain for $t \ge 0$:
\begin{eqnarray*}
\mathbb{P}\left(  \sup_{z \in A(\omega)}|z|\ge x\right) &=&\mathbb{P}\left(  \sup_{z \in A(\vartheta_t\omega)}|z|\ge x\right)\\ &\le&  \mathbb{P}_k\left( R_{K_t+1}\geq x\right)+\varepsilon\\
&=&\P_k\left( K_t \geq \log_2(x)-1\right) + \varepsilon. 
\end{eqnarray*}
Letting $t \to \infty$ and then $\varepsilon \to 0$, we obtain the first inequality.

Next, we estimate the tails of $\mu$. 
Due to Theorem \ref{th:semi_covergence}, we know that $\mu$ is of the following form
\begin{equation}
\mu(n)= C^{-1} \nu(n) m(n), \qquad n \in \{2,3,...\}\text{,}\notag
\end{equation}
where $C>0$ is a normalization constant, $\nu$ the invariant measure of the embedded Markov chain $(X_n)_{n\in\mathbb{N}}$ and $m(n)$ the mean waiting time in state $n$ as above. As mentioned before, we can simply estimate $m(n)$ by the constant $\theta$ from above. 
It remains to estimate the invariant measure $\nu(n)$ of the embedded Markov chain. 
We abbreviate $\alpha:= \frac{2}{3}\gamma -v +1 >0$. Then, by Lemma \ref{th:vanishing_probability} and Corollary \ref{co:vanishing_probability}, there exists a constant $\overline{C}$ such that
\begin{equation*}
p_i \le \overline{C}  2^{-i\alpha }
\text{.}
\end{equation*}
Hence,
\begin{equation*}
\nu(n) = Z^{-1} \prod_{i=2}^{n-1}\frac{p_i}{1-p_{i+1}} \leq  2^{-\frac{1}{2}n^2\alpha+cn}
\text{,}
\end{equation*}
for some constants $Z$ and $c$.
Therefore, for some possibly different $c$,
\begin{equation*}
\mu(n) \leq 2^{-\frac{1}{2}n^2\alpha + cn}
\text{,}
\end{equation*}
so, again for some possibly different $c>0$,
\begin{equation*}
\begin{split}
\mathbb{P}\left(  \sup_{z \in A(\omega)}|z|\ge x\right) 
&\leq \sum_{k= \lceil \log_2(x)-1\rceil \vee 2} ^\infty \mu\left(k  \right)\\
&\leq 2^{-\frac{1}{2}(\lfloor \log_2(x)-1\rfloor)^2\left(1+\frac{2}{3}\gamma-v\right)+c(\log_2x +1)}
\text{.}
\end{split}
\end{equation*}
The fact that all moments are finite is clear from the tail estimate.
\end{proof}

\begin{remark}
Recall that if the RDS generated by a Markov process admits a (weak) random attractor $A(\omega)$, then the Markov process also admits (at least one) invariant probability measure $\rho$ (but not vice versa) and, for every measurable set $B$, we have $\P\big(A(\omega) \subset B)\leq \rho(B)$ and therefore a tail estimate for the radius of $A$ automatically yields the same tail estimate for $\rho$ by choosing $B$ to be a ball of radius $R$ around the origin.
\end{remark}

\section{Conclusions}

We have proven the existence of a random attractor for a model system which, in the absence of noise, explodes in finite time. We have concentrated on finding conditions for the existence of a random attractor. We suspect that our (sufficient) conditions on the parameters are not sharp. Additionally, we have not emphasized the   stability of the one-point motion, that is, whether the trajectory starting from a single condition blows up or not. It would be interesting to complete the picture with sharp conditions for both the attractor and the stability of the one-point motion. Given the nature of the instability in the example, we suspect they are the same.

The noise only has a substantial effect in a small region in phase space. The resulting dynamics, while still random has a more predictable character than typical SDEs. One concrete manifestation of this is that the system has a fairly regular period. While it is still random, it has a much more deterministic behavior than a typical SDE with a positive rotation number. 
It would be interesting to explore the rotation number and its relationship to the parameters of the problem.


\bigskip

\noindent {\bf Acknowledgments:}   JCM thanks the NSF grant DMS-1613337 for partial support of this work and David Herzog and Avanti Athreya for many useful discussion around this general topic.  ML thanks the Duke Mathematics department for its hospitality and support during 2012/13 when the work was initiated. JCM and MS thank MSRI and the Program on "New Challenges In PDE: Deterministic Dynamics And Randomness In High And Infinite Dimensional Systems" for its hospitality and support during the fall of 2015 where this work continued. All three of us thank the referee for many useful suggestions.

\section*{Appendix A: Proof of Lemma \ref{th:double_integral}}
We prove Lemma \ref{th:double_integral}. Recall that  $A(\phi)=\phi/2+\sin(2\phi)/4$.
\begin{proof}
First note that $A$ is strictly increasing. When $K$ is large we expect the main contribution to come from the integration area where $\beta$ is around $\pi/2$ and $z$ is slightly larger.

Clearly, the integral 
\[
\int_0^\pi\int_{\beta +1}^\infty \e^{-K(A(z)-A(\beta))}\mathrm{d}z\mathrm{d}\beta 
\] 
decays exponentially fast in $K$, so it suffices to consider the remaining integral. We have
\[
\int_0^\pi\int_{\beta}^{\beta +1} \e^{-K(A(z)-A(\beta))}\mathrm{d}z\mathrm{d}\beta = \int_0^\infty \e^{-Kx}\dd m(x),
\] 
where $m$ is the image of Lebesgue measure $\lambda$ on the set $\Delta:=\{(\beta,z): \beta\leq z \leq \beta +1, \,0 \le \beta \le \pi\}$ under the map
$(\beta, z)\mapsto A(z)-A(\beta)$.  We will show that there is a constant $C >0$ such that $m([0,x])\le C\cdot x^{2/3}$ for all $x \ge 0$. Then the claim  follows:
$$
\int_0^\infty \e^{-Kx}\dd m(x) \le \int_0^\infty \e^{-Kx}C\frac 23 x^{-1/3}\,\dd x=K^{-2/3} \frac 23 C \Gamma (2/3).
$$
It remains to show that $m([0,x])\le C\cdot x^{2/3}$ for all (or for all sufficiently small) $x \ge 0$. Define
\begin{align*}
M_y:=\big\{(\beta,z):&\,0 \le \beta \le \pi,\,\beta \le z \le \beta +\kappa y^{2/3}\big\}\\ 
&\cup \big\{(\beta,z):0 \le \beta \in \big[ \frac \pi 2 -\kappa y^{1/3},
  \frac \pi 2 +\kappa y^{1/3}\big],\,z \in \big[\beta,\beta +\kappa y^{1/3}]\big\},
\end{align*}
where $\kappa >0$ will be fixed later. Clearly, 
$$
\lambda (M_y) \le \big(\pi \kappa  + 2 \kappa^2\big) y^{2/3}.
$$
Therefore our  claim follows once we have shown that there exists $\kappa >0$ such that
$$
\{(\beta,z)\in \Delta: A(z)-A(\beta)\le y\} \subseteq M_y,
$$
for all sufficiently small $y \ge 0$ and this can indeed be checked in a straightforward way. 
 \end{proof}

\section*{Appendix B: Proof of Lemma \ref{le:downfall}}

We prove Lemma \ref{le:downfall}.

\begin{proof}
Define
\[T_\varepsilon\coloneqq \inf\left\lbrace t\geq 0 \colon \sup_{0\leq s\leq t} X_s -X_t > \varepsilon \right\rbrace \qquad\text{ with } X_t=\sigma W_t +rt\text{.}\]
Then,
\begin{align*}
 &\mathbb{P}\left(\sup_{0\leq s\leq t\leq T} \left(\sigma (W_t-W_s)-r(t-s)\right)\leq \varepsilon \right)\\
&=\mathbb{P}\left(\sup_{0\leq s\leq t\leq T} \left(\sigma (W_s-W_t)+r(s-t)\right)\leq \varepsilon\right)\\
&=\mathbb{P}\left(\sup_{0\leq s\leq t\leq T} \left(X_s-X_t\right)\leq \varepsilon\right) =\mathbb{P}\left(\sup_{0\leq t\leq T}\left(\sup_{0\leq s\leq t} X_s-X_t\right)\leq \varepsilon\right)\\
&=\mathbb{P}\left(T_\varepsilon\ge T\right)
\text{.}
\end{align*}
In \cite{Taylor75} the author explicitly computed the Laplace transform of $T_\varepsilon$: for $\beta >0$,
\begin{align*}
\mathbb{E}\mathrm{e}^{-\beta T_\varepsilon} &= \frac{\delta\mathrm{e}^{-\Gamma\varepsilon}}{\delta\cosh(\delta\varepsilon)-\Gamma\sinh(\delta\varepsilon)}\text{,}\\
& \text{where } \delta = \sqrt{\left(\frac{r}{\sigma^2}\right)^2+\frac{2\beta}{\sigma^2}} \text{ and } \Gamma =\frac{r}{\sigma^2}
\text{.}
\end{align*}
In the special case $\beta=1$, we obtain
\begin{equation}\label{ch3_LaplaceTransform}
\begin{split}
\mathbb{E}\mathrm{e}^{-T_\varepsilon} &= \frac{\delta\mathrm{e}^{-\Gamma\varepsilon}}{\delta\cosh(\delta\varepsilon)-\Gamma\sinh(\delta\varepsilon)}=\frac{2}{\left(1-\dfrac{\Gamma}{\delta}\right)\mathrm{e}^{(\delta+\Gamma)\varepsilon}+\left(1+\dfrac{\Gamma}{\delta}\right)\mathrm{e}^{(-\delta+\Gamma)\varepsilon}}\\
&\leq \frac{2}{\left(1-\dfrac{\Gamma}{\delta}\right)} \mathrm{e}^{-(\delta+\Gamma)\varepsilon} \leq 2\frac{\sqrt{\Gamma^2+\frac{2}{\sigma^2}}}{\sqrt{\Gamma^2+\frac{2}{\sigma^2}}-\Gamma} e^{-2\varepsilon\Gamma}
\text{.}
\end{split}
\end{equation}
Now, we assume that $r \ge \sqrt{2}\sigma$. Then,  
\[
\sqrt{\Gamma^2+\frac{2}{\sigma^2}}-\Gamma \geq \frac{1}{2\sigma^2\Gamma} \quad \text{ and } \quad \Gamma^2\geq \frac{2}{\sigma^2},
\] 
and therefore 
\begin{align*}
\mathbb{E}\mathrm{e}^{-T_\varepsilon} \leq  2\frac{\sqrt{\Gamma^2+\frac{2}{\sigma^2}}}{\sqrt{\Gamma^2+\frac{2}{\sigma^2}}-\Gamma} e^{-2\varepsilon\Gamma}& \leq 
4\sqrt{2}\Gamma^2 \sigma^2  e^{-2\varepsilon\Gamma} = \frac{4\sqrt{2}}{\sigma^2} r^2   e^{-2\frac{\varepsilon}{\sigma^2} r}.
\end{align*}
We conclude by using Chebyshev's inequality
\begin{align*}
&\mathbb{P}\left(\sup_{0\leq s\leq t\leq T} \left(\sigma (W_t-W_s)-r(t-s)\right) > \varepsilon\right)
=\mathbb{P}\left(T_{\varepsilon}< T\right)=\mathbb{P}\left(\mathrm{e}^{- T_{\varepsilon}}> \mathrm{e}^{- T}\right)\\
&\leq \mathrm{e}^{T}\mathbb{E}\mathrm{e}^{-T_{\varepsilon}}
\leq \frac{4\sqrt{2}}{\sigma^2}\mathrm{e}^{T} r^2 \mathrm{e}^{-2\frac{\varepsilon}{\sigma^2} r}
\text{.}
\end{align*}
\end{proof}

\end{document}